\documentclass[12pt]{article}

\usepackage{graphicx}
\usepackage{amsmath,amsthm,amssymb,color,bbm}
\usepackage[noadjust,sort]{cite}
 \usepackage[normalem]{ulem}

\normalsize

\theoremstyle{plain}
\newtheorem{Theorem}{Theorem}[section] %
\newtheorem{Lemma}[Theorem]{Lemma}

\newtheorem{Corollary}[Theorem]{Corollary}

\theoremstyle{definition}
\newtheorem{Remark}[Theorem]{Remark}

\def\leq{\leqslant}
\def\geq{\geqslant}
\def\le{\leqslant}

\def\calF{\mathcal{F}}
\def\calC{\mathcal{C}}
\def\calL{\mathcal{L}}
\def\calH{\mathcal{H}}

\theoremstyle{definition}

\newtheorem{Problem}{Problem}[section]

{\hfill$\scriptstyle\blacksquare$} 

\numberwithin{equation}{section} 
\numberwithin{figure}{section} 
\numberwithin{table}{section} 

\def\supp{\operatorname{supp}}

\def\Reals{{\mathbb{R}}}
\def\Complexes{{\mathbb{C}}}

\def\Naturals{{\mathbb{N}}}

\def\dfrac#1#2{\lower0.15ex\hbox{\large$\textstyle\frac{#1}{#2}$}}

\def\one{\mathbbm{1}}

\begin{document}

\setcounter{page}{1}

\markboth{M.I. Isaev, R.G. Novikov}{H\"older-logarithmic stability in   Fourier synthesis 
}

\title{H\"older-logarithmic stability in  Fourier  synthesis \thanks{The first author's research is  supported by    the Australian Research  Council  Discovery Early Career Researcher Award DE200101045. 
}}
\date{}
\author{ 
Mikhail Isaev\\
\small School of Mathematics\\[-0.8ex]
\small Monash University\\[-0.8ex]
\small Clayton, VIC, Australia\\
\small\texttt{mikhail.isaev@monash.edu}
\and
Roman G. Novikov\\
\small CMAP, CNRS, Ecole Polytechnique\\[-0.8ex]
\small Institut Polytechnique de Paris\\[-0.8ex]
\small Palaiseau, France\\
\small IEPT RAS, Moscow, Russia\\
\small\texttt{novikov@cmap.polytechnique.fr}
}

\maketitle
\begin{abstract}
	We prove a H\"{o}lder-logarithmic stability estimate for the problem of finding a sufficiently regular compactly supported function $v$ on $\Reals^d$ from its Fourier transform  $\calF v$
	given on $[-r,r]^d$. 
	This estimate  relies on a H\"{o}lder stable continuation of $\calF v$
		from  $[-r,r]^d$  to a larger domain. The related reconstruction procedures are based on truncated series of Chebyshev polynomials. 
		 We also give an explicit example showing optimality of our stability estimates.
		
\end{abstract}

\section{Introduction}\label{S:intro}

\emph{The Fourier transform} $\mathcal{F}$ is defined by 
\begin{equation*}
 	\mathcal{F} v (\xi) := 
 	\frac{1}{(2\pi)^d}\int\limits_{\mathbb{R}^d} e^{i\xi x} v(x) dx, \ \ \ \xi\in \mathbb{R}^d,
 \end{equation*}
where $v$ is a test function on $\Reals^d$ and  $d\geq 1$.  The analysis of this transform is one of the 
most developed areas of mathematics and has many applications in physics, statistics and  engineering;
see, for example,  Bracewell \cite{Bracewell2000}.
 In particular, it is well known that if  $v$ 
is integrable and compactly supported then $\calF v $ is analytic. Thus, the Fourier transform $\calF v$ and, consequently, the function $v$ are uniquely determined  
by  the values of $\calF v $ within any open non-empty domain.  However, in the case of noisy  data,  
the reconstruction can be  hard  unless the values of $\calF v$ are known in a very large domain
 or $v$ belongs to a specific class of functions (a priory information). 
In the present paper, we study how much the stability improves with respect to  the size of the domain  where $\calF v$ is given and  with respect to the regularity of $v$.

Specifically, we consider the following   problem.
\begin{Problem}\label{Problem1}
Suppose that  $v\in \calL^1(\Reals^d)$ is 
supported in a given compact set in $\Reals^d$. 
	 The values of $\mathcal{F}v$  are given on  $[-r,r]^d$, possibly with some noise.  Find $v$.
\end{Problem}
 
 Reconstructing a compactly supported function 
from its partially known Fourier transform or, equivalently, 
computing the Fourier transform of a band-limited function given within some domain
is a classical problem of the Fourier synthesis  and image processing;  see, for example,  
\cite{LRC1987, AMS2009, BM2009, Papoulis1975, CF2014, Gerchberg1974}.
%
%
It also  arises in  studies of   inverse scattering problems in the Born approximation. 
For example, a variant of Problem~\ref{Problem1} with
$\mathcal{F}v$ given on the ball $B_ { 2\sqrt{E}}$,  where 
$B_r:= \left\{\xi\in \mathbb{R}^d :  |\xi| \le r \right\}$,
can be regarded as  a linearized inverse scattering problem for the Schr\"odinger
equation with potential $v$ at fixed positive energy $E$, for $d\geq 2 $, and on the the energy interval $[0,E]$, for $d\geq 1$.
More details can be found in the recent paper \cite[Section 4]{Novikov2020}.
This version of Problem 1.1 with $\mathcal{F}v$ given on the ball $B_ {\omega_0}$
also arises, for example, in a multi-frequency inverse source problem for the homogeneous Helmholtz equation with frequencies $\omega\in [0, \omega_0]$;
see  ~\cite[Section 3]{BLT2010} for more details.

We focus on the stability  of reconstructions for Problem  \ref{Problem1}.  In particular, for a suitable function $\phi$ such that
$ \phi(\delta) \rightarrow 0$ 
as $\delta \rightarrow 0$, we show  that
\begin{align}
			\|v_1 &- v_2\|_{\calL^2(\Reals^d)} \leq \phi\left(\|\mathcal{F}v_1 - \mathcal{F} v_2\|_{\calL^\infty([-r,r]^d)}\right), \label{type1}	
\end{align}
 under the additional assumption  that $v_1-v_2$ 
 is
 sufficiently regular (more precisely, $v_1-v_2$  is from the Sobolev space $\calH^m (\Reals^d)$ for some integer $m>0$).
The function $\phi$ also depends on a priori information about  $v_1 - v_2$.
Furthermore, we propose  a  reconstruction procedure  for Problem  \ref{Problem1}  which  stability behaviour is consistent with  the function $ \phi$.


%

It is well known in the community that Problem \ref{Problem1} is  ill-posed 
in \ the sense of Hadamard; see, for example, \cite{LRC1987, AMS2009}.
 For the introduction to the theory of ill-posed  problems; see the classical  books
  by Tikhonov, Arsenin  \cite{TA1977} and by Lavrent'ev et al.  \cite{LRS1986}.
In fact, one  can show that this problem is \emph{exponentially ill-posed}  
in a similar way to the results of \cite{Mandache2001, Isaev2013++} 
using the estimates of $\epsilon$-entropy and $\epsilon$-capacity in functional spaces  that go  back to Kolmogorov and Vitushkin~\cite{Kolmogorov1959}.
To completely settle the question,
 we give an explicit example demonstrating  exponential ill-posedness of 
Problem \ref{Problem1} in Section \ref{S:ill-posed} of the present paper.   Consequently,   a logarithmic  bound is the best one could hope to get  in \eqref{type1}, in general; see  Corollary~\ref{C:log} and Theorem~\ref{T:last}.

On the other hand, in the case when $r$ is sufficiently  large (such that $\mathcal{F} v$ on $\mathbb{R}^d \setminus [-r,r]^d$ is negligible) one can approximate $v$  
in  a Lipschitz stable manner 
 by direct computation of \emph{inverse Fourier transform}
$\mathcal{F}^{-1}$:
\begin{equation}\label{eq:first}
	v(x) \approx [\mathcal{F}^{-1} w](x) := \int \limits_{\mathbb{R}^d}w(\xi) e^{-i\xi x} d \xi,
\end{equation}
 taking  $w$ equal to the given values of  $\mathcal{F} v$ in $[-r,r]^d$, and  $w \equiv 0$, 
outside $[-r,r]^d$.   However, there remains some error in this approximation even in the absence of noise.

 In the present work, we prove  a \emph{H\"older-logarithmic stability} estimate for Problem~\ref{Problem1}
 tying together the aforementioned two facts; see Theorem \ref{Theorem2}.
%
Furthermore, our   estimate 
illustrates   similar stability behaviour in more complicated non-linear  inverse problems. 
In fact,  the relationship is  closer than a mere illustration
for the reason that the monochromatic reconstruction from the scattering amplitude  in the Born approximation reduces to Problem \ref{Problem1} with $\calF v$ given on the ball $B_r$ for some 
$r>0$. In connection with this reduction, 
the logarithmic stability estimates  of the present work  
can be compared with 
  results of 
\cite{HH2001,IN2013++, HW2017}. 
For other known similar stability results  in related inverse problems,  see, for example, \cite{Alessandrini1988, BLT2010, Isaev2013+, IN2014, Isakov2011, Novikov2011, Santacesaria2015}  and references  therein.
However, to our knowledge,    H\"older-logarithmic or logarithmic stability estimates for Problem \ref{Problem1} 
  are  implied by none of the  results in the literature before the present work.

%
%

The main idea of our stable reconstruction for Problem \ref{Problem1}  is the following. First,
 we extrapolate $\calF v$ from $[-r,r]^d$ to a larger domain, which size depends on the noise level. Then, we apply the inverse Fourier transform. This leads to our second problem.
\begin{Problem}\label{Problem2}
Suppose that  $v\in \calL^1(\Reals^d)$ is  supported in a given compact set in $\Reals^d$. 
	 The values of $\mathcal{F}v$  are given on  $[-r,r]^d$, possibly with some noise.
Find $\mathcal{F} v$ on $[-R,R]^d$, where $R\geq r$.
\end{Problem}

 Problem \ref{Problem2} is equivalent to the  band-limited  extrapolation (for $d=1$), which was considered in many   works; see, for example,
 \cite{BM2009, Cadzow1979 ,CF2014, Gerchberg1974, Papoulis1975}.   A more general problem of  stable analytic continuation of a complex function was considered  in  \cite{DT2019, LRS1986, Tuan2000, Vessella1999}. In particular, 
 \cite[Theorem 1.2]{DT2019}  or  \cite[Theorem 1]{Vessella1999} 
 lead to a  \emph{H\"older stability} estimate for Problem \ref{Problem2}:    for some $0<\alpha<1$ and $c_{\alpha,R}>0$, 
 \begin{equation}\label{0type1}
	\begin{aligned}
		\|\mathcal{F}v_1 - \mathcal{F}v_2\|_{L^\infty([-R,R]^d)} \leq c_{\alpha,R} \Big(\|\mathcal{F}v_1 - \mathcal{F} v_2\|_{L^\infty([-r,r]^d)}\Big)^{\alpha}.
	\end{aligned}
\end{equation}
However,  for a fixed $\alpha$, the factor   $c_{\alpha, R}$   in this estimate grows exponentially  as $R$ increases, which  hinders continuation of $\calF v|_{[-r,r]^d}$ to very large domains. This behaviour is natural due to exponential ill-posedness of Problem \ref{Problem1}.

 In this paper, we continue these studies by establishing  H\"older stability  estimates  for  a multidimensional  extrapolation procedure  for   Problem~\ref{Problem2}  based on the Chebyshev polynomials; see  formula \eqref{0type2}  and Theorem  \ref{Theorem3}. It is essential that we give these estimates with explicit constants. 
   Then, by choosing an appropriate $R$, we apply this  result to Problem   \ref{Problem1};
   see   formula~\eqref{type2}, Corollary~\ref{Corollary_C}, 
   and  the proof of Theorem \ref{Theorem2}.
  
%
%
%
%

For a fixed $r>0$, we consider  the following    family of   extrapolation procedures  $\mathcal{C}_{R,n}[\cdot]$ depending on two parameters 
$R\geq r$ and 
$n \in \mathbb{N}:=\{0,1,\ldots\}$.
%
 For a  function $w$  on $[-r,r]^d$,  define
\begin{equation}\label{def_C}
	{\mathcal C}_{R,n} [w](\xi) := 
	\begin{cases}
	 w(\xi), & \xi \in [-r,r]^d,\\\displaystyle
	\sum \limits_{\substack{
	 k_1,\ldots,k_n \in \Naturals \,:
	  \\
	 k_1+\cdots+k_d<n}
}
	a_{k_1,\ldots, k_d}[w] 
	\prod_{j=1}^d 
		T_{k_j}\left(\dfrac{\xi_j}{r}\right), & \xi \in [-R,R]^d\setminus [-r,r]^d,
	\\
		0,	& \xi \in \mathbb{R}^d\setminus [-R,R]^d,
	\end{cases}
\end{equation}
 where
\begin{equation}\label{def_a}
	a_{k_1,\ldots, k_d}[w] :=
	\int_{-r}^{r}\cdots \int_{-r}^{r} w(\xi) \prod_{j=1}^d \left(\frac{2^{ \one[k_j>0]} T_{k_j} \left(\frac{\xi_j}{r}\right)}
	{\pi (r^2 - \xi_j^2)^{\frac12}} \right) d\xi_1 \ldots d\xi_d 
\end{equation}
and
  ${\mathcal C}_{R,n} [w]$  is taken to be $0$ everywhere outside $[-r,r]^d$ in the case when $n=0$.
In the above,
$\one[k>0]$ is the indicator function for 
$\{k>0\}$:  
\[
 \one[k>0]= \begin{cases} 
 1,& \text{if $k>0$,} \\
 0,& \text{otherwise};
\end{cases}
\]
 and $(T_k)_{k\in \mathbb{N}}$ stand  for \emph{the Chebyshev polynomials},  which can be defined  
by   $T_k(t):= \cos(k\operatorname{arccos}(t))$  for $t\in [-1,1]$ and extended  to $|t|>1$ in a natural way.


Recall that if $v$ is integrable and  compactly supported then
 $\mathcal{F}v$  is analytical in $\Complexes^d$.  It follows that, for all $\xi\in \Complexes^d$,
 \begin{equation} \label{eq:series}
  	\mathcal{F} v (\xi) =  \sum_{k_1 =0}^\infty \cdots \sum_{k_d=0}^\infty a_{k_1,\ldots,k_d} \left [\mathcal{F}v|_{[-r,r]^d}\right] 
  	\prod_{j=1}^d T_{k_j}\left(\dfrac{\xi_j}{r}\right);
  \end{equation}      
  see, for example, Lemma \ref{L:coeff} and inequality \eqref{eq:4last}.
  We will show that if  $w \approx \mathcal{F}v|_{[-r,r]^d}$ then ${\mathcal C}_{R,n} [w](\xi)$   approximates well the  series of \eqref{eq:series}   in the region 
   $[-R,R]^d \setminus [-r,r]^d$, provided $n$ is sufficiently  large  so the tail of the series is negligible, but not very large so the continuation $\mathcal{C}_{R,n}$  is sufficiently stable. 
%
%

 
  In Section \ref{S:con} and  Section \ref{S:rec}, we prove the new stability estimates for Problem \ref{Problem2} and Problem \ref{Problem1}, respectively.  We state them in the following forms.
 \begin{itemize}
 \item  Let   $\|\mathcal{F}v - w\|_{\calL^\infty([-r,r]^d)} \leq \delta$.
 Then,   for any $R>r$,  there is some $n^* = n^*(\delta, R)$  and $0<\alpha<1$ such that,
 \begin{equation} \label{0type2}
 	\|\mathcal{F}v -  \mathcal{C}_{R,n^*} [w]\|_{\calL^\infty([-R,R]^d)} \leq   c_{\alpha,R} \, \delta^\alpha.
 \end{equation}
 \item In addition, if $v$ is sufficiently regular (more precisely, $v$  is from  the Sobolev  space $\calH^m (\Reals^d)$ for some integer $m>0$), then there are  some $R(\delta)$ and $n(\delta)= n^*(\delta, R(\delta))$
 such that, as  $ \delta \rightarrow 0$,
 \begin{equation}\label{type2}
		\|v - \mathcal{F}^{-1} \mathcal{C}_{R(\delta), n(\delta)} [w]\|_{\calL^2(\Reals^d)} 
		\leq \phi(\delta)\rightarrow 0.
		%
\end{equation}
\end{itemize}
The constant $c_{\alpha,R}$ and the function $\phi$   depend on a priori information about $v$.
 Note that \eqref{type2} and  \eqref{0type2}  imply  
 \eqref{type1} and \eqref{0type1}, respectively, by setting  $v := v_1-v_2$ and $w := 0$, and using the linearity of the considered problems and the reconstruction procedures.

  In Section~\ref{S:L:Proof}, we prove a technical lemma  about   $\mathcal{C}_{R,n}$.  In Section~\ref{S:ill-posed}, we give an example demonstrating that our  estimates (at least, for Problem~\ref{Problem1} at fixed $r$) are essentially
optimal; see Theorem \ref{T:last}. Finally,  in   Section~\ref{S:directions},
we outline several directions for further development of the studies of the present work.



\section{H\"older stability  in Problem \ref{Problem2}}\label{S:con}

 In this section, we give stability estimates for the extrapolation procedure   $\mathcal{C}_{R,n}$ defined according to \eqref{def_C}; see Lemma \ref{Lemma_C}, Theorem \ref{Theorem3}, and Corollary \ref{Corollary_C}.  
 
 Assume that the unknown function
  $v:\Reals^d \rightarrow \Complexes$ is  such that,
 for some  $N, \sigma>0$,
   \begin{equation} \label{eq:ass}
\|v\|_{\calL^1(\Reals^d)}\leq (2\pi)^d N,\qquad 
     \supp(v) \subseteq  \left\{x \in \Reals^d : \sum_{j=1}^d|x_j| \leq \sigma\right\}.
  \end{equation}   
 Assume also that  the given data $w : [-r,r]^d \rightarrow \Complexes$ is such that, for some $\delta>0$,
\begin{equation}\label{eq:ass2}
     \|w - \calF v\|_{\calL^\infty([-r,r]^d)} \leq \delta  <N, 
   \end{equation}
where $\mathcal{F}$ is  the Fourier transform.

 Note that
if   \eqref{eq:ass} holds then, for any $\xi \in \Reals^d$,
\begin{equation}\label{eq:N-delta}
	|\calF v(\xi)| \leq \frac{1}{(2\pi)^d} \int_{\Reals^d}|v(x)|dx \leq N.
\end{equation}
This explains the condition $\delta<N$ in assumption \eqref{eq:ass2}. Indeed, if the noise level $\delta$ is greater than $N$ then the given data $w$ tells  about  $v$  as little as the trivial function $w_0\equiv 0$.



First, we give an estimate for $\mathcal{C}_{R,n}$ for arbitrary  $R$ and $n$. 
\begin{Lemma}\label{Lemma_C}
 Let the assumptions of \eqref{eq:ass} and \eqref{eq:ass2}  hold for some $N, \delta, r, \sigma>0$.   
Then,     for  any integer $n>0$ and real $\rho,R>0$ such that $R\geq r$, $\rho \geq 4R/r$,   we have
	\begin{equation*}\label{eq:Lemma_C}
		\|\mathcal{F}v - \mathcal{C}_{R,n}[w] \|_{\calL^{\infty}([-R,R]^d)} 
		\leq  \frac{1}{4} \left(4^d  \left(\dfrac{4R}{r}\right)^n \delta  +  \left(\dfrac{16}{3}\right)^d Ne^{r\sigma\rho}  \left(\dfrac{4R}{3r \rho}\right)^n\right).
	\end{equation*}
\end{Lemma}

 Lemma  \ref{Lemma_C} is proved in Section  \ref{S:L:Proof}. 
Optimising the parameter $n$ in Lemma \ref{Lemma_C}, we obtain the following H\"older stability estimate for Problem \ref{Problem2}.

\begin{Theorem}\label{Theorem3}
 Let the assumptions of \eqref{eq:ass} and \eqref{eq:ass2}  hold for some $N, \delta, r, \sigma>0$. 
 Assume that $\rho,R>0$ are such that  $R\geq r$ and   $\rho\geq 4R/r$.
Then, we have
\begin{equation*}\label{eq:Theorem3}
		\|\mathcal{F}v - \mathcal{C}_{R,n^*}[w] \|_{\calL^{\infty}([-R,R]^d)} 
		\leq \left(\dfrac{16}{3}\right)^d \dfrac{R}{r}    \left( \dfrac{N e^{r\sigma \rho}}{\delta}\right)^{\tau(\rho)} \delta,
	\end{equation*}
	where 
		\begin{equation*}
		\begin{aligned}
		n^* := 
		 \left\lceil \frac{ \ln\frac{N}{\delta} + r\sigma \rho }{ \ln (3\rho)}\right\rceil   
		\qquad \text{and} \qquad
		\tau(\rho) :=  \frac{\ln \frac{4R}{r}}{\ln(3\rho) }.
		\end{aligned}
	\end{equation*}
\end{Theorem}
%
\begin{proof}
Using  \eqref{eq:ass2}, we have that
\[
 	\eta :=    \frac{ \ln\frac{N}{\delta} + r\sigma \rho }{ \ln (3\rho)}>0.
\] 
By definition,  we find that $\eta \leq n^* < \eta+1$
and  $\delta = N    e^{r \sigma \rho} (3\rho)^{-\eta}$.
Using that $R\geq r$, we get
\begin{align*}
		  \delta \left(\dfrac{4R}{r}\right)^{\eta+1}        
		  = 
		  \dfrac{  4R}{r   }  N    e^{r \sigma \rho} \left(\dfrac{4R}{3r\rho}\right)^{\eta}
		 \geq    4  N e^{r \sigma \rho} \left(\dfrac{4R}{3r\rho}\right)^{\eta}.
\end{align*}
Then, applying Lemma \ref{Lemma_C}, we obtain that
\begin{equation*}
	\begin{aligned}
		\|\mathcal{F}v - \mathcal{C}_{R,n^*}[w] \|_{\calL^{\infty}([-R,R]^d)} 
		&\leq \dfrac14 \left( 4^d \left(\dfrac{4R}{r}\right)^{n^* } \delta  + \left(\dfrac{16}{3}\right)^d N  e^{r \sigma \rho} \left(\dfrac{4R}{3r\rho}\right)^{n^* }  \right)
	\\	&\leq \dfrac14 \left( 4^d \left(\dfrac{4R}{r}\right)^{\eta+1 } \delta  + \left(\dfrac{16}{3}\right)^d N  e^{r \sigma \rho} \left(\dfrac{4R}{3r\rho}\right)^{\eta}  \right)
		 \\ &\leq \left(\dfrac{16}{3}\right)^d \dfrac{R}{r}  \left(\dfrac{4R}{r}\right)^{\eta} \delta \left(\left(\dfrac34\right)^d + \dfrac{1}{4}\right) \leq   \left(\dfrac{16}{3}\right)^d\dfrac{R}{r}  \left(\dfrac{4R}{r}\right)^{\eta} \delta.
	\end{aligned}
\end{equation*}
Since $\tau(\rho) \ln (3 \rho)= \ln \frac{4R}{r}$,  we get 
\[
		   \left(\dfrac{4R}{r}\right)^{\eta} \delta =  	 (3\rho)^{\tau(\rho) \eta } \delta= 
	   \left(N e^{r\sigma \rho}\right)^{\tau(\rho)} \delta^{1-\tau(\rho)}.
\]
Combining the above estimates completes the proof.
\end{proof}

Next, 
to 
achieve optimal stability  bounds for Problem \ref{Problem1},
  the parameters $R$  and $n$  in the reconstruction $\calF^{-1}\calC_{R,n}$    have to be chosen carefully depending on  a priory information; see formulas \eqref{def_L}, \eqref{def_Cstar}, \eqref{def_R}
  and Corollary \ref{Corollary_C}.
   For any $\tau \in [0,1]$, let
\begin{equation}\label{def_L}
			L_{\tau}(\delta) = L_\tau(N,\delta,r,\sigma):= \max\left\{1,\, \dfrac{1}{4}\left(\dfrac{(1-\tau) \ln  \frac{N}{\delta}}{r\sigma}\right)^\tau\right\}.
	\end{equation}
 Here and thereafter, we assume  $0<\delta<N$. 
Using \eqref{def_C}, define 
	\begin{equation}\label{def_Cstar}
		\calC^*_{\tau,\delta} := {\mathcal C}_{R_\tau(\delta),n_\tau(\delta)},
	\end{equation}
	 where 
	\begin{equation}\label{def_R}
		\begin{aligned}
		R_\tau(\delta) &= R_\tau(N, \delta,r,\sigma):= r L_{\tau}(\delta),\\
		n_\tau(\delta) &= n_\tau(N,\delta,r, \sigma):= 
		\begin{cases}
		\displaystyle \left\lceil\dfrac{ (2-\tau)\ln\frac{N}{\delta}}{ \ln 3 + \frac{1}{\tau}\ln (4L_{\tau}(\delta))}
		\right\rceil, 
		&\text{ if    $\tau>0$,}\\
		0,&\text{otherwise.}
		\end{cases}
		\end{aligned}
	\end{equation}
	and $\lceil\cdot\rceil$ denotes the ceiling of a real number.

Theorem \ref{Theorem3} leads to the following stability estimate for   $\calC^*_{\tau,\delta} $;  
 which will be  crucial for the results of the next section; see  Theorem \ref{Theorem2}. 
\begin{Corollary}\label{Corollary_C}
 Let the assumptions of \eqref{eq:ass} and \eqref{eq:ass2}  hold for some $N,\sigma, r, \delta>0$.  
Then, for any $\tau \in [0,1]$,  we have
 	\begin{equation*}
		\|\mathcal{F}v - \calC^*_{\tau,\delta} [w] \|_{\calL^{\infty}\left([-R_\tau(\delta), R_\tau(\delta)]^d\right)} 
		\leq    \left(\dfrac{16}{3}\right)^d     N  \left(\dfrac{\delta}{N}\right)^{(1-\tau)^2}L_{\tau}(\delta),
	\end{equation*}
where $L_{\tau}(\delta)$ and $R_\tau (\delta)$  are defined in \eqref{def_L} and \eqref{def_R}.
\end{Corollary}
\begin{proof} 
First, we consider the case $L_{\tau}(\delta) = 1$, for which  $R_\tau(\delta) = r$. 
Recalling from \eqref{eq:ass2} that $\delta<N$ and using \eqref{def_C} for $R=r$, we find that
\begin{align*}
	\|\calF v - \calC^*_{\tau,\delta}w\|_{\calL^{\infty}\left([-R_\tau(\delta), R_\tau(\delta)]^d\right)}
	\leq \delta \leq  \left(\dfrac{16}{3}\right)^d N  \left(\dfrac{\delta}{N}\right)^{(1-\tau)^2}L_{\tau}(\delta).
\end{align*}

Next, suppose that 
\[	
L_{\tau}(\delta) = \dfrac14\left(\dfrac{(1-\tau) \ln \frac{N}{\delta}}{r\sigma}\right)^\tau>1.
\]
This implies that  $\tau>0$. Let
	$
		\rho := (4L_{\tau}(\delta))^{1/\tau}.
	$
	Then, we get $e^{r \sigma \rho} = \left(\frac{N}{\delta}\right)^{1-\tau}$ and, by the assumptions, 
	\begin{align*}
		R_\tau (\delta)\geq r  \qquad \text{and} \qquad \rho \geq   4L_{\tau}(\delta) =   4R_\tau(\delta)/r.  
			\end{align*} 
	 Applying Theorem \ref{Theorem3} 
	 and observing that $n^*$ coincides with $n_\tau(\delta)$ defined by \eqref{def_R}, we get that
		\begin{align*}
		\|\mathcal{F}v - \calC^*_{\tau,\delta} [w] \|_{\calL^{\infty}\left([-R_\tau(\delta), R_\tau(\delta)]^d\right)} 
		\leq
			 \left(\dfrac{16}{3}\right)^d  L_{\tau}(\delta)   \left( \left(\dfrac{N}{\delta}\right)^{2-\tau}\right)^{\tau(\rho)}  
			\delta,
		\end{align*}
		 where $\tau(\rho)$ is defined in Theorem \ref{Theorem3}.  
		 Note that  $\tau(\rho)$ is  different from $\tau$. However, we can replace
$\tau(\rho)$ by $\tau$ in the estimate above
since  $ \delta < N$ and
 	\[
	\tau(\rho)  = \frac{\ln(4R_\tau(\delta)/r)}{\ln (3\rho)} = 
	\frac{\ln (4L_{\tau}(\delta))}
	{\ln 3 + \frac{1}{\tau} \ln  (4L_{\tau}(\delta))} \leq \tau.
	\]
	The required bound follows.
\end{proof}


\section{H\"older-logarithmic stability in  Problem \ref{Problem1}}\label{S:rec}

To prove our stability estimate for Problem \ref{Problem1},   in addition to \eqref{eq:ass}, we assume also  that $v \in \calH^{m}(\Reals^d)$,
where  $\calH^{m}(\Reals^d) $  is the standard Sobolev space of $m$-times smooth functions in $\calL^2$ on $\mathbb{R}^d$.  
Consider the  seminorm $|\cdot|_{\calH^{m} (\Reals^d)}$
in  $\calH^{m}(\mathbb{R}^d)$  defined by
\begin{equation}\label{def_sem}
	|v|_{\calH^{m} (\Reals^d)} := 
	\left(\sum_{ j =1}^d   
	\left\| \dfrac{\partial^m v}{(\partial x_j)^m}\right\|_{\calL^2(\Reals^d)}^2\right)^{1/2}.
\end{equation}

\begin{Theorem} \label{Theorem2}
Let the assumptions of  \eqref{eq:ass}  and \eqref{eq:ass2}  hold for some $N,\delta, r, \sigma>0$. 
Assume also that	 $v \in \calH^m(\mathbb{R}^d)$ for some   integer $m>0$
and $|v|_{\calH^{m}(\mathbb{R}^d)} \leq \gamma$  for some $\gamma>0$.
  Then, for any $\tau \in [0,1]$,	    the following  holds:
	 \begin{equation} \label{eq_Theorem2}
		\begin{aligned}
		\|v - \mathcal{F}^{-1}\calC^*_{\tau,\delta}[w] \|_{\calL^2(\mathbb{R}^d)}
		 \leq  \left(20 \sqrt{r}\right)^d        N \left(L_{\tau}(\delta)\right)^{d/2+1}  \left(\dfrac{\delta}{N}\right)^{(1-\tau)^2}
		 +   \gamma
		  \left( rL_\tau(\delta)\right)^{-m}.
		 \end{aligned}
		 \end{equation}
\end{Theorem}

The first term of the right-hand side in  estimate \eqref{eq_Theorem2} corresponds
  to the error caused by the H\"older stable continuation of the noisy data $w$  from $[-r,r]^d$ to $[-R_{\tau}(\delta),R_{\tau}(\delta)]^d$ and the second (logarithmic) term corresponds to the error  caused by ignoring the values of  $\calF v$ outside  
  $[-R_{\tau}(\delta),R_{\tau}(\delta)]^d$; see the proof  of  Theorem \ref{Theorem2} for details.

\begin{Remark}
Clearly, the stability behaviour in Problem \ref{Problem1} should not depend on scaling of functions or arguments.  
It might be obscure at first sight,  but estimate \eqref{eq_Theorem2} is invariant with respect to such scalings. Indeed,   for some $\alpha,\beta>0$, let    
$\tilde v$ be defined by $\tilde{v}(x) := \alpha v(\beta   x)$,    $x\in \Reals^d$.  The Fourrier transforms of $v$ and $\tilde v$ satisfy the following relation   
$\calF \tilde{v}(\xi) = \alpha \calF v(\beta^{-1}\xi)$ for $\xi \in \Reals^d$.
If $w\approx \calF v$ in $[-r,r]^d$ then, equivalently, $\tilde{w} \approx \calF \tilde{v}$ in $[-\tilde{r},\tilde{r}]^d$, where $\tilde{r} = \beta r$
and  $\tilde{w}(\xi) := \alpha w(\beta^{-1}  \xi)$,  $\xi\in \Reals^d$. The other parameters  in \eqref{eq:ass}  and \eqref{eq:ass2}  are modified as follows:
$\tilde{N}=\alpha N$, $\tilde{\delta} = \alpha \delta$,  and $\tilde\sigma =   \beta^{-1} \sigma.$  Observe that $L_{\tau}(\delta)$  depends only on $r\sigma$
and  $N/\delta$, which are independent of scalings. Finally, we have
\begin{align*}
	\|\tilde v - \mathcal{F}^{-1}\calC^*_{\tau,\delta}[ \tilde w] \|_{\calL^2(\mathbb{R}^d)} &=
	\alpha \beta^{d/2} \|v - \mathcal{F}^{-1}\calC^*_{\tau,\delta}[w] \|_{\calL^2(\mathbb{R}^d)},
	\\
	|\tilde v|_{\calH^m(\Reals^d)} &= \alpha \beta^{m+d/2} |v|_{\calH^m(\Reals^d)}. 
\end{align*}
Thus, both sides of   estimate  \eqref{eq_Theorem2}  get multiplied by the same constant $\alpha \beta^{d/2}$, that is,  the statements of Theorem \ref{Theorem2} for $v$,$w$ and for $\tilde{v}$, $\tilde{w}$ are equivalent. 
\end{Remark}

\begin{proof}[Proof of Theorem \ref{Theorem2}]
%

 Let all assumptions of  Theorem \ref{Theorem2} hold. 
The Parseval-Plancherel identity states that  
		\begin{equation}\label{Parseval}
		 \|u\|_{\calL^2(\mathbb{R}^d)} = (2\pi)^{d/2}\|\mathcal{F} u \|_{\calL^2(\mathbb{R}^d)} = (2\pi)^{-d/2}\|\mathcal{F}^{-1} u \|_{\calL^2(\mathbb{R}^d)}. 
		\end{equation}
Thus, we get that 
\begin{equation*}
	\begin{aligned}
	\|v - \mathcal{F}^{-1} \calC^*_{\tau,\delta} [w]\|_{\calL^2(\mathbb{R}^d)} 
	\leq (2\pi)^{d/2} (I_1+I_2),
	\end{aligned} 
\end{equation*}
where
\begin{equation*}
	\begin{aligned}
		I_1 &:= \left(\int_{[-R_\tau(\delta),R_\tau(\delta)]^d }  \left|\mathcal{F}{v}(\xi) -		\calC^*_{\tau,\delta} [w](\xi)\right|^2 d\xi \right)^{1/2}, 
\\
	I_2 &:= \left(\int_{\Reals^d \setminus [-R_\tau(\delta),R_\tau(\delta)]^d }  \left|\mathcal{F}{v}(\xi) \right|^2 d\xi \right)^{1/2}. 
	\end{aligned}
\end{equation*}
	 Using Corollary \ref{Corollary_C}, we get that, 
	 \begin{equation*}	
	 	\begin{aligned}
	 	I_1 &\leq   \left( \int_{[-R_\tau(\delta),R_\tau(\delta)]^d} 
	 	 \left\|\mathcal{F}{v} - \calC^*_{\tau,\delta}w\right\|^2_{\calL^\infty([-R_\tau(\delta),R_\tau(\delta)]^d)} d\xi  \right)^{1/2} \\
	 	&\leq    \left(\dfrac{16}{3}\right)^d     N  \left(\dfrac{\delta}{N}\right)^{(1-\tau)^2} L_{\tau}(\delta) \left(2 R_\tau(\delta)\right)^{d/2}
	 	\\
	 	& \leq   \left(20 \sqrt{\dfrac{r}{2\pi}}\right)^d       N \left(L_{\tau}(\delta)\right)^{d/2+1}  \left(\dfrac{\delta}{N}\right)^{(1-\tau)^2}.
	 	\end{aligned}
	 \end{equation*} 
Next,   applying  \eqref{Parseval} and recalling  the seminorm $|\cdot|_{\calH^{m}(\mathbb{R}^d)}$ defined in \eqref{def_sem}, we find that
\begin{equation*}
	 \sum_{j=1}^d \|\xi_j^m \mathcal{F}{v}\|_{\calL^2(\Reals^d)}^2  = \frac{1}{(2\pi)^d}\sum_{j=1}^d 
	 \left\|\frac{\partial^m v} {(\partial x_j)^m}\right\|_{\calL^{2}(\Reals^d)}^2  =    \frac{ |v|_{\calH^{m}(\mathbb{R}^d)}^2}{(2\pi)^d}.
\end{equation*}	 
%
%
Since  $\Reals^d \setminus [-R_{\tau}(\delta),R_{\tau}(\delta)]$ is covered by  the regions
$\varOmega_j:=\{\xi \in \Reals^d:  |\xi_j|> R_\tau(\delta)\}$, for $j=1,\ldots,d$,  we obtain that
 \begin{equation*}
	\begin{aligned}
	I_2&\leq \left( \sum_{j=1}^d 
	\int_{|\xi_j|> R_\tau(\delta)}  \left|\frac{\xi_j^m \calF v(\xi)}{ (R_\tau(\delta))^{m}}\right|^2 d\xi \right)^{1/2} 
	\\
	&\leq  \left( \sum_{j=1}^d  \frac{\|\xi_j^m\calF v\|_{\calL^2(\Reals^d)}^2}{(R_\tau(\delta))^{2m}  } 
   \right)^{1/2}   \leq   \frac{|v|_{\calH^{m}(\mathbb{R}^d)}
}{(2\pi)^{d/2}}  \left( r  L_{\tau}(\delta)  \right)^{-m}.	\end{aligned}
\end{equation*}
Combining the above bounds for $I_1$ and  $I_2$, we  complete the proof of Theorem \ref{Theorem2}.
\end{proof}

Theorem  \ref{Theorem2}  leads to the H\"older-logarithmic stability estimate 
for $\|v_1 - v_2\|_{\calL^2(\Reals^d)}$ in \eqref{type1}, provided $v_1-v_2$ satisfies assumptions   \eqref{eq:ass}, \eqref{eq:ass2} (for fixed $N$)
and $v_1-v_2 \in \calH^m(\Reals^d)$,  as explained in Section \ref{S:intro} after formula \eqref{type2}.
%
%
Estimate \eqref{eq_Theorem2} with $\tau=0$ is similar to well-known stability results for approximate reconstruction  explained in \eqref{eq:first}.
Theorem  \ref{Theorem2}  also implies the following corollary.
\begin{Corollary}\label{C:log}
Let  $v: \Reals^d \rightarrow \Complexes$ be  supported in a compact  set $A \subset \Reals^d$,   where $d \geq 1$, 
and $ |v |_{\calH^{m}(\Reals^d)}  \leq \gamma $ for some integer $m>0$  and  real $\gamma>0$.
Let  $B$ 
be an open set in $\Reals^d$  and  
$\|\calF v  \|_{\calL^{\infty}(B)}  <1$.
 	Then, 
 	for any  $0\leq \mu<m $ , there is  $ c=  c(A,B,\gamma, \mu,m)>0$ such that 
	\begin{equation}\label{log-2}
 		\|v \|_{\calL^{2}(\Reals^d)} \leq    c \left(\ln \frac{1}{\|\calF v  \|_{\calL^{\infty}(B)}}  \right)^{-\mu}.
 	\end{equation}
\end{Corollary} 
\begin{proof}
 Without loss of generality, we can assume that $0 \in B$ by considering   $\tilde{v} := v  e^{i\xi_0 x}$, for a fixed $\xi_0 \in \Reals^d$.
Then, $[-r,r]^d \subset B$ for a sufficiently small $r$. 
Any compact set $A$ lies in $\left\{x \in \Reals^d : \sum_{j=1}^d|x_j| \leq \sigma\right\}$
for a sufficiently large $\sigma$. 
Since $v$ is compactly supported,  the condition  $|v |_{\calH^{m}(\Reals^d)}  \leq \gamma$,  for $m\geq 1$,
implies that the norm $\|v\|_{\calL^{1}(\Reals^d)}$ is bounded above by $(2\pi)^{d}N$, where the constant $N$  depends on $A$, $m$, $\gamma$ only. 
 Applying
  Theorem~\ref{Theorem2} with $\tau:=\mu/m$, $w\equiv 0$, 
  $
    \delta :=\|\calF v  \|_{\calL^\infty([-r,r]^d)} 
  $  and observing that the logarithmic term dominates  
the H\"older term as $\delta \rightarrow 0$ and  
\[
	\|\calF v  \|_{\calL^\infty([-r,r]^d)} \leq \|\calF v  \|_{\calL^{\infty}(B)}<1,
\]
we complete the proof.
\end{proof}

In Section \ref{S:ill-posed}, we show that the exponent $\mu$ in Corollary \ref{C:log} is optimal (or almost optimal for $d=1$),  using  an explicit construction. 
Namely, we prove that, for $d\geq 2$ and any $\mu>m$,  there is some $v$  violating   \eqref{log-2}
no matter how large  constant
$c$    we take.   For $d=1$, the same holds for any $\mu>m+1/2$.  The optimality of the threshold exponent $\mu^* = m$ for the case $d=1$ remains an open question. 
Note that our instability examples also show  an optimality of \eqref{eq_Theorem2}  as a logarithmic stability estimate.



\section{Proof of Lemma \ref{Lemma_C}}\label{S:L:Proof}

%

 To prove Lemma \ref{Lemma_C}, we need the bounds for  series of
  Chebyshev polynomials stated in the following lemma.  
  We will use the standard combinatorial fact that the number of ways  to write $n$ as a sum of 
  $d$ nonnegative integers (ordered) equals  the binomial coefficient
  \begin{equation}\label{eq:binom}
     \binom{n+d-1}{n}= \frac{(n+d-1)!}{n! (d-1)!}.
  \end{equation} 
\begin{Lemma}\label{L:coeff}
Let $\sigma, r, N>0$ and $R\geq r$. 
 If   $v$ satisfies  \eqref{eq:ass} 
then  the following holds. 
\begin{itemize}
\item[(a)]For any $\rho\geq  1$,  $\xi \in [-R,R]^d$  and $k_1,\ldots,k_d \in \Naturals$,  we have
 \begin{equation*}
 	\left|a_{k_1,\ldots,k_d}\left[\calF v|_{[-r,r]^d}\right] \prod_{j=1}^d 
		T_{k_j}\left(\dfrac{\xi_j}{r}\right)\right| \leq     2^d Ne^{ \frac12 r\sigma\rho}
		 \left(\frac{2R}{r\rho}\right)^{\sum_{j=1}^d k_j},
 \end{equation*}
 where $\calF$ is the Fourier transform and $a_{k_1,\ldots,k_d}[\cdot]$ is defined according   to \eqref{def_a}.
\item[(b)]
For any  $\rho' \geq  4R/r$, we have 
\[
	 \left\| \calF v   - \calC_{R,n} \left[\calF v|_{[-r,r]^d}\right] 
	 \right\|_{\calL^\infty([-R,R]^d)}
	 \leq     \left(\frac83\right)^d Ne^{  r\sigma\rho'}  \binom{n+d-1}{n}\left(\frac{4R}{3r \rho'}\right)^n,
\] 
where $ \calC_{R,n}[\cdot]$ is defined according  to \eqref{def_C}.
\end{itemize}
\end{Lemma}
\begin{proof}
%
 For $z_1,\ldots,z_d \in \Complexes$, let 
  \[ 
       f(z_1,\ldots,z_d) := \calF v(r\cos z_1, \ldots, r \cos z_d).
  \]
   Observe that,  for  any $z\in \Complexes$ ,
	\[
		|\Im (\cos z)| \leq \dfrac12 |e^{\Im z} - e^{-\Im z}| \leq \dfrac12 e^{|\Im z|},
	\]
	where $\Im z$ denote the imaginary part of $z$.
  If $|\Im  z_j| \leq  \ln \rho$ for all $1\leq j\leq d$, then,  by assumptions,    for any $x \in \supp(v)$,  we find that 
   \[
   	\left|\sum_{j=1}^d x_j \Im (\cos z_j)\right| \leq  \sum_{j=1}^d |x_j| \rho/2 \leq \dfrac12 \sigma  \rho.
   \]
   Therefore,
    \begin{align*}
 	|f(z_1,\ldots,z_d)| &= \left| \frac{1}{(2\pi)^d} \int_{\Reals^d}
 	e^{i \sum_{j=1}^d r  x_j \cos z_j } v(x) dx\right|\\
	&\leq \frac{1}{(2\pi)^d} \int_{\supp(v)}   e^{\frac12 r \sigma \rho} |v(x)| dx =  Ne^{\frac12 r\sigma\rho}.
 \end{align*}
  Observing that $f$ is $2\pi$-periodic even function with respect to each component and recalling  definition \eqref{def_a}  and that  $T_k(t):= \cos(k\operatorname{arccos}(t))$  for $t\in [-1,1]$, we get
 \[
 	  	  a_{k_1,\ldots,k_d}=	 \frac{2^{\sum_{j=1}^d \one[k_j>0]}}{(2\pi)^d}
  	\int_{0}^{2\pi}
  	\cdots \int_{0}^{2\pi}  e^{i \sum_{j=1}^d k_j\varphi_j }f(\varphi_1,\ldots,\varphi_d) d\varphi_1\ldots d\varphi_d.
 \]
       Since $v$ is compactly supported,  
 we have that $\calF v$ and $f$ are  analytic functions  in $\Complexes^d$. 
  Using the Cauchy integral theorem,  we estimate
  \begin{align*}
  	&\left|\frac{1}{(2\pi)^d}
  	 \int_{0}^{2\pi} e^{i \sum_{j=1}^d k_j\varphi_j }f(\varphi_1,\ldots,\varphi_d) d\varphi_1\ldots d\varphi_d \right|
  	\\ 
  	&\qquad \qquad =   	\left|\frac{1}{(2\pi)^d}
  	\int_{i\ln \rho}^{2\pi+i\ln \rho}\cdots 	\int_{i\ln \rho}^{2\pi+i\ln \rho} e^{i \sum_{j=1}^d k_j z_j }f(z_1,\ldots,z_d) dz_1\ldots dz_d \right|
  	\\
  	& \qquad \qquad \leq   
  	\frac{1}{(2\pi)^d}
  		\int_{i\ln \rho}^{2\pi+i\ln \rho}
  	\cdots 	\int_{i\ln \rho}^{2\pi+i\ln \rho}
  	 Ne^{\frac12 r\sigma\rho} e^{-\sum_{j=1}^d k_j \ln\rho} dz_1\ldots dz_d\\
  	  	& \qquad \qquad = Ne^{\frac12 r\sigma\rho}  \rho^{-\sum_{j=1}^d k_j }.
  \end{align*}
 We complete the proof of  part (a), by observing that  
$
|T_k(t)| \leq (2R/r)^k 
$
for any $|t| \leq R/r$.
 Indeed,  if $|t|\leq 1$ then $|T_k(t)|\leq 1$, otherwise
\[
  \begin{aligned}
 	|T_k(t)| = |\cosh (k\operatorname{arccosh}(t))| = \dfrac{1}{2}\left|(t-\sqrt{t^2-1})^k + (t+\sqrt{t^2-1})^k\right|
 	\leq (2|t|)^k.
   \end{aligned}
\]

  For (b),  let $\rho:= 2\rho'$ and $\lambda := \dfrac{2R}{r\rho} = \dfrac{R}{r \rho'}  \leq \dfrac14$.  
  Using the Taylor theorem with the remainder in the Lagrange form,   we get that,   for some $\lambda'\in[0,\lambda]$,
  \begin{align*}
  	(1-\lambda)^{-d} -  \sum_{k=0}^{n-1} 
	\binom{k+d-1}{k} \lambda^{k}&=  
	\binom{n+d-1}{n} (1-\lambda')^{-d-n} \lambda^n 
	\\
	&\leq  \binom{n+d-1}{n}  \left(\frac{4}{3}\right)^d \left(\frac{4 \lambda }{3}\right)^n.
  \end{align*}
   Using   \eqref{eq:binom}  and part (a), we find that
\begin{align*}
 &\left\| \calF v   - \calC_{R,n} \left[\calF v|_{[-r,r]^d}\right] 
	 \right\|_{\calL^\infty([-R,R]^d)}
	\leq \sum_{k=n}^{\infty}
	\sum_{k_1+\cdots +k_d =k }    2^d  Ne^{\frac12 r\sigma\rho}  \lambda^{k} 
	\\ &\qquad  \qquad\qquad \qquad = 	 2^{d}   Ne^{r\sigma\rho'} \left( (1-\lambda)^{-d} - 
	\sum_{k=0}^{n-1} 
	\binom{k+d-1}{k} \lambda^{k}\right) \\
	&\qquad \qquad\qquad \qquad \leq   	   \left(\frac83\right)^d Ne^{ r\sigma\rho'}  \binom{n+d-1}{n}  \left(\frac{4 \lambda }{3}\right)^n.
\end{align*}
This  completes the proof of  Lemma \ref{L:coeff}.
 \end{proof}

Now we are ready to proceed to  Lemma \ref{Lemma_C}.
Recall that $|T_k(t)| \leq  1$, if $|t|\leq 1$.
	Using \eqref{def_a} and the assumptions,  we find  that,  for any $k_1,\ldots,k_d \in \Naturals$,
	\begin{align*}
		 &\left|a_{k_1,\ldots, k_d}[w] -a_{k_1,\ldots, k_d} \left[\mathcal{F} v|_{[-r,r]^d} \right] \right| =
		 \left| a_{k_1,\ldots, k_d}\left[ w - \mathcal{F} v|_{[-r,r]^d} \right]\right|
		\\ &\qquad \qquad  \qquad \qquad \leq 
	\int_{-r}^{r}\cdots \int_{-r}^{r} \delta \prod_{j=1}^d \left(\frac{2^{ \one[k_j>0]} \left|T_{k_j} \left(\frac{\xi_j}{r}\right)\right|}
	{\pi (r^2 - \xi_j^2)^{\frac12}} \right) d\xi_1 \ldots d\xi_d
	\leq 2^d \delta.
	\end{align*}
	Recalling also  that $|T_k(t)| \leq (2|t|)^k$ for $|t|\geq 1$, we get 
	\begin{align*}
		 \left\| \mathcal{C}_{R,n}[w] - \mathcal{C}_{R,n}\left[\mathcal{F} v|_{[-r,r]^d}\right] 	
		 \right\|_{\calL^{\infty}([-R,R]^d)}
		 	&\leq \sum_{k=0}^{n-1} \sum_{k_1+\cdots+k_d =k} 2^d \delta \left(\frac{2R}{r}\right)^k
		 	\\=  2^d \delta \sum_{k=0}^{n-1} \binom{k+d-1}{k}\left(\frac{2R}{r}\right)^k
		 	&\leq  2^d \delta \binom{n+d-1}{n}\left(\frac{2R}{r}\right)^n.
	\end{align*}
	Since $n\geq 1$ and $d\geq1$, we have that 
	\[ 
	\binom{n+d-1}{n} = \binom{n+d-2}{n-1} +\binom{n+d-2}{n}  \leq \binom{n+d-1}{n-1} +\binom{n+d-1}{n+1},
	\] where 
	$\binom{n+d-2}{n}$ and $\binom{n+d-1}{n+1}$
	 are taken to be $0$ if $d=1$. Thus,  
	we get
	\begin{equation}\label{eq:4last}
	 \binom{n+d-1}{n} \leq \frac12 \sum_{j=0}^{n+d-1} \binom{n+d-1}{j} =  2^{n+d-2}.
	\end{equation}
		  Combining the above and using  Lemma \ref{L:coeff}(b),
		  	   we complete the proof of Lemma \ref{Lemma_C}.
	

\section{Exponential ill-posedness of Problem \ref{Problem1}}\label{S:ill-posed}

In this section,   we prove that Problem \ref{Problem1} is exponentially ill-posed.  For ease of presentation, we employ the asymptotic notations 
$O(\cdot)$ and $\Omega(\cdot)$ always
 referring to the passage of the parameter $n$ to infinity.
For two sequences of real numbers $a_n$ and $b_n$, we say
$a_n=O(b_n)$ if there exist constants $C>0$ and $n_0\in \Naturals$ such that $|a_n|\le C\,|b_n|$ for all 
$n>n_0$.
We say 
$a_n=\Omega(b_n)$  if $a_n>0$ always and $b_n = O(a_n)$.

First, we consider  an explicit function   $v_{n,m} :\Reals^2 \rightarrow \Complexes$   similar to  the one given  by Mandache \cite[Theorem 2]{Mandache2001}.
Let $g \in C^{\infty}(\Reals)$ 
 be a nontrivial function supported in a compact set  of positive real numbers. For example, one can take 
 \begin{equation}\label{def_g}
 	g(t):= 
 	\begin{cases}
 		\exp\left(\frac{1}{(t-1)(t-2)}\right), &\text{if } 1<t<2,\\
 		0, & \text{otherwise.}
 	\end{cases}
 \end{equation}
  For  integer $n \geq 1$ and $m\geq 0$, let $v_{n,m}$  be defined by
\begin{equation*}
	v_{n,m}(x_1,x_2) := n^{-m} e^{i n\varphi}  g(t),
\end{equation*}
where $t\geq 0$, $\varphi \in [0,2\pi)$, and $(x_1,x_2) =(t \cos \varphi, t\sin\varphi)$.  Observe  that, as $n\rightarrow \infty$, 
\begin{equation}\label{vnm1}	
	\|v_{n,m}\|_{\calL^{2}(\Reals^2)} = \Omega( n^{-m}).
\end{equation}
It is also straightforward  that 
\begin{equation}\label{vnm2}
	 \|v_{n,m}\|_{C^m(\Reals^2)} = O(1);
\end{equation}
see, for example, the arguments of \cite[Theorem 2]{Mandache2001}. 
\begin{Lemma}\label{l:vnm}
For any  $m\in \Naturals$ and $r>0$, we have
$
 \|\calF  [\Re v_{n,m} ]\|_{\calL^{\infty}([-r,r]^2)} = O(e^{-n}).
$
\end{Lemma}
\begin{proof}
%
Writing the Fourier transform in the polar coordinates, we find that
\[
	\calF v_{n,m} (\xi) = \frac{n^{-m}}{(2\pi)^2} \int_{\supp (g)}   t g(t)  \left( \int_{0}^{2\pi} e^{i t |\xi| \cos(\varphi-\varphi_0)} e^{in\varphi} d\varphi\right) dt,
\]
where $\xi = (|\xi|\cos \varphi_0, |\xi|\sin \varphi_0)^T \in \Reals^2$. 
Using the Cauchy integral theorem, we get that, uniformly over all  $\xi \in [-r,r]^2$ and $t \in \supp(g)$,
\[	
	\int_{0}^{2\pi}  e^{i t |\xi| \cos (\varphi - \varphi_0) } e^{in\varphi} d\varphi
	= O\left(
	\int_{0+i}^{2\pi+i}  e^{i t |\xi| \cos (z - \varphi_0) } e^{inz} dz
	\right) = O(e^{-n}).
\]
 Observing also  $\calF[\Re v_{n,m}](\xi) = \calF v_{n,m} (\xi)+\calF v_{n,m} (-\xi)^*$,  where $z^*$ denotes the complex conjugate of $z\in \Complexes$, the required bound follows. 
\end{proof}

The following  theorem  implies that the exponent $\mu$ in Corollary  \ref{C:log} is  optimally bounded above by $m$ (or almost optimally, for $d=1$) since 
$  |v|_{\calH^m(\Reals^d)} \leq C \|v\|_{C^m(\Reals^d)}$ 
for a compactly supported $v$, where  $C$  depends on $\supp(v)$   only.
\begin{Theorem}\label{T:last}
 Let $d\geq 1$ and  $m\geq 0$ be integers. 
Let $\mu$ be a positive real number satisfying either $\mu> m$   if  $d\geq 2$, 
or $\mu >m+1/2$  if $d=1$.  Then, 
for any bounded open set $A \subset \Reals^d$, compact set $B \subseteq \Reals^d$, and positive constants $\gamma, c$, there exists $v:\Reals^d \rightarrow \Reals$ such that:
\begin{align}
	\supp(v)  \subseteq  A, \qquad      &\|v\|_{C^m(\Reals^d)} \leq \gamma, \qquad 
	\|\calF v \|_{\calL^\infty(B)}< 1,
	 \nonumber  \\
	 \|v\|_{\calL^{2}(\Reals^d)} &> c\left( \ln  \frac{1}{\|\calF v \|_{\calL^\infty(B)}}\right)^{-\mu}. \label{T:est}
\end{align}
\end{Theorem} 
\begin{proof}
First, we consider the case $d\geq 2$.
Define $w_{n,m}: \Reals^d \rightarrow \Complexes$ by
\[
w_{n,m}(x) := \Re v_{n,m}(x_1,x_2)  \prod_{j=3}^d  g(x_j),
\]
 where $g$ is  given in \eqref{def_g}.  Observe that  $w_{n,m}\in C^m(\Reals^d)$ and is   compactly supported.
  Using \eqref{vnm2} and  taking any  $x_0 \in A$ and  sufficiently small   $\alpha>0$ and sufficiently big  $\beta>0$, we get that 
   the functions $v_n:\Reals^d \rightarrow \Reals$ defined by
 \[
    v_n(x):=\alpha w_{n,m}\left(\beta(x - x_0)\right)
 \]
 are supported in $A$ and 
  satisfy
 $ \|v_n\|_{C^m(\Reals^d)} \leq \gamma$ for all  $n>0$. 
 Next, taking   $r$ to be sufficiently large and  observing from  \eqref{def_g} that $g$ is supported in $ [1,2]$ and $|g(t)|\leq 1$ for all $t\in \Reals$,   we ensure
 \[
 	                                  \left\|\calF v_n\right\|_{\calL^\infty(B )}
 	                                 = O\left(  \left\|\calF[ \Re v_{n,m}] \right\|_{\calL^\infty([-r,r]^2)} \right). 
 \]
 Using \eqref{vnm1} and Lemma \ref{l:vnm}, we get that, 
 as $n \rightarrow \infty$,
 \[
 	 \|v_n\|_{\calL^{2}(\Reals^d)} =  \Omega(n^{-m}) \qquad \text{and} \qquad 
 	  \left\|\calF  v_n  \right\|_{\calL^\infty(B )} = O(e^{-n}).
 \]
Taking $v\equiv v_n$ for sufficiently large $n$, we get \eqref{T:est}. 
  
For the case $d=1$, consider the functions  $h_{n,m}:\Reals \rightarrow \Complexes$ defined by
 \[
 	h_{n,m} (x) := \int_{\Reals} \Re v_{2n,m} (t,x) dt
 	= \int_{-2}^{2} \Re v_{2n,m} (t,x) dt.
 \]
  From \eqref{vnm2}, we  derive that  
 \[
 	\|h_{n,m}\|_{C^m(\Reals)} = O(1).
 \]
 Using Lemma \ref{l:vnm}, we also  find  that, for any fixed $r>0$, 
 \[
 	\| \calF  h_{n,m}\|_{\calL^\infty([-r,r])} = 2\pi \|\calF [\Re v_{2n,m}](0,\cdot)\|_{\calL^\infty([-r,r])}
   = O(e^{-n}).
 \]
  Note that if $|x| \leq (2n)^{-1}$ then, by the definition of $v_{n,m}$,
 \[
 	h_{n,m}(x) \geq n^{-m} \left(2 \cos 1 \int_{1}^2 g(t) dt  + O(n^{-1}) \right) = \Omega(n^{-m}).
 \] 
Therefore, 
\[
\|h_{n,m}\|_{\calL^2(\Reals)} \geq \left(n^{-1}  \min_{|x|\leq (2n)^{-1}}  |h_{n,m}(x)|  \right)^{1/2} = \Omega(n^{-m-1/2}).
\]
We complete the proof by considering functions of the form $\alpha h_{n,m}\left(\beta^{-1}(x-x_0)\right)$ 
and
 repeating the arguments of the case $d\geq 2$. 
\end{proof}

\section{Further developments}\label{S:directions}

Studies of the present work can be developed, in particular, in the following directions. These issues will be addressed in further articles.

\subsection{Estimates in $\calH^s$}
 
 Theorem \ref{Theorem2} and Corollary \ref{C:log} admit analogs 
 in the norm  $\|\cdot\|_{\calH^s(\Reals^d)}$   in place of $\|\cdot\|_{\calL^2(\Reals^d)}$ in the left-hand sides of \eqref{eq_Theorem2} and 
 \eqref{log-2}. 
   Recall that, for real $s$,
  the Sobolev space $\calH^s(\Reals^d)$ and its norm can be defined by
\begin{equation*}\label{def_Hm}
	\begin{aligned}
	\calH^{s}(\mathbb{R}^d) &:= \left\{u \in \calL^2(\mathbb{R}^d):\ \mathcal{F}^{-1} (1+|\xi|^2)^{s/2} \mathcal{F} u \in  \calL^2(\mathbb{R}^d) \right\},
\\
	||u||_{\calH^{s}(\mathbb{R}^d)} &:= \left\|\mathcal{F}^{-1} (1+|\xi|^2)^{s/2} \mathcal{F} u \right\|_{\calL^2(\mathbb{R}^d)}.
\end{aligned}
\end{equation*}
  For example,
 in a similar way with \eqref{eq_Theorem2} and 
 \eqref{log-2}, one can estimate
 $		\|v - \mathcal{F}^{-1}\calC^*_{\tau,\delta}[w] \|_{\calH^s(\Reals^d)}$
  and  $\|v\|_{\calH^s(\Reals^d)}$ for arbitrary real 
  $s$ and $m$ such that $s<m$ and $m>-\dfrac{d}{2}$. 
    In particular,  under assumptions 
    \eqref{eq:ass}, \eqref{eq:ass2}, and  $\|v\|_{\calH^m(\Reals^d)} \leq \gamma$,
    one can show that 
  \begin{equation}\label{eq:appod}
  	\|v - \mathcal{F}^{-1}\calC^*_{\tau,\delta}[w] \|_{\calH^s(\Reals^d)} 
  	\leq c \left(\ln (3+\delta^{-1})\right)^{-\tau(m-s)}
  \end{equation}
  for any $\tau \in (0,1)$ and some constant $c = c(N, \sigma, r, m, s, \gamma,  \tau,d)>0.$

 The results of the  type \eqref{eq:appod} with $s<0$ can be used for apodized reconstructions based on  
 $\calC^*_{\tau,\delta}$. Indeed,  let $\phi$ be a real-valued compactly supported function on $\Reals^d$ and $\hat{\phi} = \calF \phi$. We assume that $\phi$ satisfies the following: 
 \[
 \int_{\Reals^d} \phi(x) dx =1;
 \]
 and, for some $\ell>0$,
 \[
 N_\ell = N_\ell[\phi] :=\sup_{\xi \in \Reals^d}   (1+|\xi|^2)^{\ell/2}  |\hat{\phi}(\xi)|<+\infty.
 \]
One may also assume, for example, that $\phi$ is non-negative, spherically symmetric, and non-increasing in $|x|$.  

Let $0<p<\ell$ and $s:=p-\ell$.  Then, for any $u \in  \calH^s(\Reals^d)$,  we   estimate
\begin{equation}\label{eq:phi-u}
\begin{aligned}
	\left\|\phi *u   \right\|_{\calH^p(\Reals^d)}  
	&= \left\|  \calF^{-1} \left[ (1+|\xi|^2)^{p/2} \hat{\phi}  \hat{u}\right] \right\|_{\calL^2(\Reals^d)} 
  \\
  &\leq c(\ell)  \left\|  \calF^{-1}[ (1+|\xi|^2)^{s/2} \hat{u}] \right\|_{\calL^2(\Reals^d)}
  =  N_\ell  \left\| u   \right\|_{\calH^s(\Reals^d)},
\end{aligned}
\end{equation}
where $*$ denote the convolution operator and $\hat{u}:=\calF u$. 
Combining \eqref{eq:appod}  and \eqref{eq:phi-u} with $u:=  v - \mathcal{F}^{-1}\calC^*_{\tau,\delta}[w]$, we obtain that 
\begin{equation}\label{last:apo}
	\left\|\phi * v -  \phi*\mathcal{F}^{-1}\calC^*_{\tau,\delta}[w]  \right\|_{\calH^p(\Reals^d)}  
	\leq c  N_{\ell}   \left(\ln (3+\delta^{-1})\right)^{-\tau (m  - s)},
\end{equation}
  under the same assumptions  on $v$,  $m$, $s = p-\ell$ as  in  \eqref{eq:appod}.  
  Estimate \eqref{last:apo} shows the logarithmic stability in  $\calH^p(\Reals^d)$
  for   the regularized (apodized) reconstruction
   $\tilde{v}:=\phi*v$ from  the given data $w$.   For more details about regularized (apodized) reconstructions for problems similar to Problem \ref{Problem1},  see, for example, \cite{AMS2009} and references therein.

%
%
  
  \subsection{Super-exponential decay}
  
  The results of the present works also admit anologs for the case when condition \eqref{eq:ass}
  is replaced by the condition
  \begin{equation} \label{N:eq:ass}
   Q_v(\lambda) :=   \dfrac{1}{(2\pi)^d}  \int\limits_{\Reals^d} e^{\lambda|x|}|v(x)|dx \leq  N\exp\left(\sigma \lambda^{\nu}\right), 
   \qquad 
   \text{for all $\lambda\geq 0$,}
  \end{equation}   
  where constants $N>0$, $\sigma>0$ and $\nu\geq 1$ are given a priori. 
  The case  $\nu=1$ corresponds  to   compactly supported $v$.

  In particular, under assumptions 
  \eqref{eq:ass2}, \eqref{N:eq:ass}, and  $\|v\|_{\calH^m(\Reals^d)} \leq \gamma$,
    one can show that
  \begin{equation}\label{eq:appod}
  	\|v - \mathcal{F}^{-1}\calC^*_{\tau,\delta}[w] \|_{\calL^2(\Reals^d)} 
  	\leq c \left(\ln (3+\delta^{-1})\right)^{-\tau m},
  \end{equation}
   for any $\tau \in \left(0,1-\sqrt{1-(1-\alpha) \nu^{-1}}\right)$ and some constant $c = c(N, \sigma, \nu,
  r,m,  \gamma,   \tau,d)>0.$

%

  \subsection{PSWF approach}
  Theorem \ref{T:last} and Corollary \ref{C:log} show that  the logarithmic estimate \eqref{log-2} is optimal  with respect to the exponent $\mu$.  Thus, estimate  \eqref{eq_Theorem2} is  also optimal in this logarithmic sense.  However, we believe that  
     \eqref{eq_Theorem2}   can be  improved  
 with respect to  dependence on a priori parameters 
 $N$, $r$, $\sigma$,  $\gamma$.    A natural approach is to employ  more advanced basis in place of Chebyshev polynomials and  exponents of the inverse Fourier transform.
     For the case when $v$ is compactly supported like in \eqref{eq:ass},  one can use, for example, the basis of the prolate spheroidal wave functions (PSWFs).  
     
     The PSWFs   $(\Psi_{k,\sigma})_{k\in \Naturals}$ can be  defined  as the eigenfunctions of the spectral problem 
     \[
     		\int_{-1}^1  \frac{\sin (\sigma(x-y))}{\pi(x-y)} \psi(y) dy = \lambda \psi(x), 
     \]
     where $\psi \in \calL^2(\Reals)$ and $\lambda$ is the spectral parameter. The functions $(\Psi_{k,\sigma})_{k\in \Naturals}$ form simultaneously an orthogonal basis in $\calL^2([-1,1])$
     and an orthogonal basis in the subspace of $\calL^2(\Reals)$ consisting of the functions  whose Fourier transform is supported in the interval $[-\sigma,\sigma]$. For more information about 
     PSWFs, see, for example, \cite{Slepian1983, BK2014} and references therein.





\begin{thebibliography}{99}

\bibitem{Alessandrini1988}
 G. Alessandrini, 
  Stable determination of conductivity by boundary measurements, 
 \textit{Applicable Analysis}, \textbf{27} (1988), 153--172.


\bibitem{AMS2009}  
N. Alibaud, P. Mar\'echal,  Y. Saesor. 
A variational approach to the inversion of truncated
fourier operators. \textit{Inverse Problems}, \textbf{25}(4) (2009), 045002.

%

\bibitem{BLT2010}
G.~Bao, J.~Lin, F.~Triki,
A multi-frequency inverse source problem,
\textit{Journal of Differential Equations}, \textbf{249} (2010), 3443--3465.


\bibitem{BM2009}
G. Beylkin, L. Monz\'on,
Nonlinear inversion of a band-limited Fourier transform, \textit{Applied and Computational Harmonic Analysis}, \textbf{27}(3) (2009), 351--366.







\bibitem{BK2014}
A.~Bonami,  A.~Karoui,  
Uniform bounds of prolate spheroidal wave functions and eigenvalues decay. 
\textit{C. R. Math. Acad. Sci. Paris} \textbf{352} (3) (2014),  229--234.


\bibitem{Bracewell2000}
R.N. Bracewell, The Fourier Transform and its Applications, 3rd edition, McGraw-Hill, Boston (2000), 616pp.




\bibitem{Cadzow1979}
J. A. Cadzow. An extrapolation procedure for band-limited signals.  \textit{IEEE Transactions on Acoustics, Speech, and Signal Processing},  \textbf{27}(1) (1979), 4--12.

\bibitem{CF2014}
E. J. Cand\`es, 
C. Fernandez-Granda, Towards a mathematical theory of super-resolution,
\textit{Communications on Pure and Applied Mathematics}, \textbf{67} (2014), 906--956.


 
%


\bibitem{DT2019}
L. Demanet,  A. Townsend,
Stable Extrapolation of Analytic Functions,
\textit{Foundations of Computational Mathematics}, \textbf{19} (2019), 297--331.




%
%

\bibitem{Gerchberg1974}
R. W. Gerchberg, 
Superresolution through error energy reduction, 
\textit{Optica Acta: International Journal of Optics}, \textbf{21}(9) (1974), 709--720.
















\bibitem{HH2001}
P. H\"ahner, T. Hohage, New stability estimates for the inverse acoustic inhomogeneous medium problem and applications, 
\textit{SIAM Journal on Mathematical Analysis},  \textbf{33}(3) (2001), 670--685.

\bibitem{HW2017}
T.~Hohage, F.~Weidling, Variational source conditions and stability estimates for inverse electromagnetic medium scattering problems,
\textit{Inverse Problems and Imaging},  \textbf{11}(1) (2017), 203--220.





\bibitem{Isaev2013+}
M. Isaev,  Energy and regularity dependent stability
estimates for near-field inverse scattering in  multidimensions, \textit{Journal of Mathematics}, (2013), Article ID 318154, 10~pp.; DOI:10.1155/2013/318154.

\bibitem{Isaev2013++}
M. Isaev, Exponential instability in the inverse scattering problem on the energy interval,
\textit{Functional analysis and its applications},  \textbf{47}(3)  (2013), 187--194.


%


%

\bibitem{IN2013++}
M. Isaev, R.G. Novikov,  New global stability estimates for monochromatic
inverse acoustic scattering, \textit{SIAM Journal on Mathematical Analysis}, \textbf{45}(3) (2013), 1495--1504.

\bibitem{IN2014}
M. Isaev, R.G. Novikov,  Effectivized H\"older-logarithmic stability estimates for the Gel'fand inverse problem, \textit{Inverse Problems}, \textbf{30} (9) (2014),  095006.


\bibitem{Isakov2011}
V. Isakov, Increasing stability for the  Schr\"odinger potential from the Dirichlet-to-Neumann map,
\textit{Discrete and Continuous Dynamical Systems},  \textbf{4}(3)  (2011), 631--640.

%

\bibitem{Mandache2001}
 N. Mandache,
{Exponential instability in an inverse problem for the Schr\"odinger equation}, 
 \textit{Inverse Problems}, \textbf{17}  (2001), 1435--1444.  





\bibitem{LRC1987}
A. Lannes, S. Roques,  M.-J. Casanove. 
Stabilized reconstruction in signal and image
processing: I. partial deconvolution and spectral extrapolation with limited field. 
\textit{Journal of
modern Optics},  \textbf{34}(2) (1987), 161--226.

%

%
\bibitem{LRS1986}
M.M. Lavrent'ev,  V.G. Romanov, S.P. Shishatskii, 
Ill-posed problems of mathematical physics and analysis, Translated from the Russian by J. R. Schulenberger.   \textit{Translations of Mathematical Monographs}, \textbf{64} (1986).
American Mathematical Society,  Providence,  R.I.,  vi+290 pp.


\bibitem{Kolmogorov1959}
A.N. Kolmogorov, V.M. Tikhomirov, 
$\epsilon$-entropy and $\epsilon$-capacity in functional spaces, \textit{Uspekhi Matematicheskikh Nauk}, \textbf{14}(2) (1959)
3--86 (in Russian). English Translation,   \textit{American Mathematical Society Translations: Series 2}, \textbf{17} (1961), 277--364.





%
%
%


%


 

%


%




\bibitem{Novikov2011}
R.G. Novikov,
New global stability estimates for the Gel'fand-Calderon inverse problem, 
\textit{Inverse 
Problems},  \textbf{27} (2011), 015001,  21pp.
%

\bibitem{Novikov2020} R.G. Novikov, Multidimensional inverse scattering for the Schr\"odinger equation, Book series:
\textit{Springer Proceedings in Mathematics and Statistics.} Title of volume: Mathematical Analysis, its Applications and Computation - ISAAC 2019, Aveiro, Portugal, July 29-August 2; Editors: P. Cerejeiras, M. Reissig (to appear), e-preprint:
https://hal.archives-ouvertes.fr/hal-02465839v1.

%
%



\bibitem{Papoulis1975}
A. Papoulis, 
A new algorithm in spectral analysis and band-limited extrapolation. 
\textit{IEEE Transactions on Circuits and Systems},  \textbf{22}(9) (1975), 735--742.

%



 
 
 
 
\bibitem{Santacesaria2015}
M. Santacesaria, 
{A Holder-logarithmic stability estimate for an inverse problem in two dimensions},
\textit{Journal of Inverse and Ill-posed Problems}, \textbf{23}(1) (2015), 51--73.


\bibitem{Slepian1983}
D.~Slepian, 
Some comments on Fourier analysis, uncertainty and modeling,
\textit{SIAM Rev.} \textbf{25 }(1983), 379--393.

%

%

\bibitem{TA1977}
A.N.~Tikhonov,  V. Y.~Arsenin,  Solutions of ill-posed problems. 
  Washington : New York :  Winston ; distributed solely by Halsted Press, (1977), 258 p.


\bibitem{Tuan2000}
V.K.~Tuan, Stable analytic continuation using hypergeometric summation,
\textit{Inverse Problems}, \textbf{16}(1) (2000), 75--87.





\bibitem{Vessella1999}
S.~Vessella, A continuous dependence result in the analytic continuation
problem, Forum Mathematicum, \textbf{11}(6) (1999), 695--703.

\end{thebibliography}
\end{document}